\documentclass[11pt]{amsart}
\usepackage{pifont}
\usepackage{amsfonts}
\usepackage{amsfonts,latexsym,rawfonts,amsmath,amssymb,amsthm}
\usepackage[plainpages=false]{hyperref}

\usepackage{graphicx}

\numberwithin{equation}{section}

\newcommand{\beq}{\begin{equation}}
\newcommand{\eeq}{\end{equation}}
\newcommand{\beqs}{\begin{eqnarray*}}
\newcommand{\eeqs}{\end{eqnarray*}}
\newcommand{\beqn}{\begin{eqnarray}}
\newcommand{\eeqn}{\end{eqnarray}}
\newcommand{\beqa}{\begin{array}}
\newcommand{\eeqa}{\end{array}}

\def\lra{\longrightarrow}

\def\bc{\begin{center}}
\def\ec{\end{center}}

\def\begeq{\begin{equation}}
\def\endeq{\end{equation}}
\def\and{\quad{\rm and}\quad}

\let\lra=\longrightarrow

\def\mapright\#1{\,\smash{\mathop{\lra}\limits^{\#1}}\,}

\newtheorem{prop}{Proposition}[section]
\newtheorem{theo}[prop]{Theorem}
\newtheorem{lem}[prop]{Lemma}

\newtheorem{defi}[prop]{Definition}
\newtheorem{conj}[prop]{Conjecture}

\title  {A note on  the $K-$stability on toric manifolds}

\author {Bin $\text{Zhou}^*$}
\address{ * Department of Mathematics, Peking University,
Beijing, 100871, China}
\author { Xiaohua $\text{Zhu}^{*,**}$}
\thanks { **  Partially supported by  NSF10425102 in China.}
 \email{xhzhu@math.pku.edu.cn}

\begin{document}
\bibliographystyle{plain}

\begin{abstract}In this note, we prove that on polarized toric manifolds
the relative $K$-stability with respect to Donaldson's toric
degenerations is a necessary condition for the existence of Calabi's
extremal metrics, and also we show that the modified $K$-energy is
proper in the space of $G_0$-invariant K\"ahler metrics  for the
case of toric surfaces which admit the extremal metrics.
\end{abstract}

\maketitle

\setcounter{section}{-1}

\section { Introduction }

Around the existence of Calabi's extremal metrics, there is a
well-known conjecture (cf. [Ya], [Ti]):

\begin{conj} A polarization  K\"ahler manifold
 $M$ admits a K\"ahler metric with constant scalar curvature (or more generally an extremal
metric) if and only if $M$ is stable in a certain sense of GIT.
\end{conj}

For  ``only if'' part of this conjecture, the first breakthrough was
made by Tian [Ti]. By introducing the concept of $K$-stability, he
gave an answer to ``only if'' part for K\"ahler-Einstein manifolds.
A remarkable progress was made by Donaldson [D1] who showed the
Chow-Mumford stability for a polarized K\"ahler manifold with
constant scalar curvature when the holomorphic automorphisms group
$\text{Aut}(M)$ of $M$ is finite. The Donaldson's result was lately
generalized by Mabuchi [M1], [M2] to any polarized K\"ahler manifold
$M$ which admits an extremal metric without any assumption for
$\text{Aut}(M)$. However it is still unknown whether there is a
generalization of Tian's result for the $K$-stability on the
K\"ahler-Einstein manifolds analogous to Donaldson-Mabuchi's result
for the Chow-Mumford stability.

As we know the definition of $K$-stability on a polarization
K\"ahler manifold is associated  to degenerations (or  test
configurations in the sense of Donaldson [D2]) on the underlying
manifold. In order to study the relation between the $K$-stability
and $K$-energy on a polarized  toric manifold, Donaldson  in [D2]
introduced  a class of special degenerations induced by  rational,
piecewise linear functions, called toric degenerations, and proved
that for the surfaces' case the $K$-energy is bounded from below
in the space of
 $G_0$-invariant K\"ahler metrics under the assumption of
$K$-stability for any toric degeneration, where $G_0$ is a maximal
compact torus group. In this note, we focus on polarized toric
manifolds as in [D2] and shall give an answer to ``only if'' part in
Conjecture 0.1 in the sense of relative $K$-stability with respect
to toric degenerations for the extremal metrics. Furthermore, we
show that the modified $K$-energy is proper in the space of
$G_0$-invariant K\"ahler metrics on a toric surface which admits an
extremal metric. The relative $K$-stability is a generalization of
$K$-stability which was first introduced by Sz\'{e}kelyhidi [Sz], as
well as the modified $K$-energy is a generalization of the
$K$-energy.

\section {Statement of main theorems}

  Let $(M,L)$ be an n-dimensional polarized K\"ahler manifold.
 Recall the definition of relative $K$-stability.

\begin {defi}[ {[D2], [Sz]} ] A test configuration for  the  polarized K\"ahler manifold $(M,L)$
of exponent $r$ consists of a $\Bbb C^*$-equivariant flat family of
schemes $\pi: \mathcal {W}\longrightarrow \Bbb C$ (where $\Bbb C^*$
acts on $\Bbb C$ by multiplication) and a $\Bbb C^*$-equivariant
ample line bundle $\mathcal {L}$ on $\mathcal {W}$. We require that
the fibres $(\mathcal {W}_t, \mathcal {L}|_{\mathcal {W}_t})$ are
isomorphic to $(M,L^r)$ for any  $t\neq 0$. A test configuration is
called trivial if $\mathcal {W}=M \times \Bbb C$ is a product.

Let  $(M, L)$ be  equipped with  $\Bbb C^*-$actions  $\beta$.  We
say that a test configuration $(\mathcal {W}, \mathcal {L})$ is
compatible with $\beta$, if there are $\Bbb C^*-$actions
$\tilde{\beta}$ on $(\mathcal {W}, \mathcal {L})$ such that
$\pi:\mathcal {W}\longrightarrow \Bbb C$ is
$\tilde{\beta}$-equivariant with trivial $\Bbb C^*-$actions on $\Bbb
C$ and the restriction of $\tilde{\beta}$ to $(\mathcal {W}_t,
\mathcal {L}|_{\mathcal {W}_t})$ for nonzero $t$ coincides with that
of $\beta$ on $(M, L^r)$ under the isomorphism.
\end {defi}

Note that  $\Bbb C^*$-actions on $\mathcal {W}$  induce   $\Bbb
C^*$-actions on the central fibre $M_0=\pi^{-1}(0)$ and the
restricted line bundle $\mathcal {L}|_{M_0}$. We denote by
$\tilde{\alpha}$ and $\tilde{\beta}$ the induced $\Bbb C^*-$actions
of $\alpha$ and $\beta$ on $(M_0,\mathcal {L}|_{M_0})$,
respectively. The relative $K$-stability is based on the modified
Futaki invariant on the central fibre,
 \beqn
F_{\tilde{\beta}}(\tilde{\alpha})= F(\tilde{\alpha})-
\frac{(\tilde{\alpha}, \tilde{\beta})}{(\tilde{\beta},
\tilde{\beta})}F(\tilde{\beta}),\eeqn
  where
$F(\tilde{\alpha})$ and $F(\tilde{\beta})$ are generalized Futaki
invariants of $\tilde{\alpha}$ and $\tilde{\beta}$ defined in [D2],
respectively, and $(\tilde{\alpha}, \tilde{\beta})$ and
$(\tilde{\beta}, \tilde{\beta})$ are inner products defined in [Sz].

\begin {defi} [{[Sz]}]   The  polarized K\"ahler manifold  $(M,L)$ with $\Bbb
C^*$-actions $\beta$ is $K$-semistable relative to $\beta$ if
$F_{\tilde{\beta}}(\cdot)\leq 0$ for any test-configuration
compatible with $\beta$. It is called relative $K$-stable in
addition that
 the equality holds if and only if the test-configuration is trivial.
\end{defi}

An n-dimensional polarized toric manifold $M$ corresponds to an
integral   polytope $P$ in $\mathbb {R}^n$ which is described by a
common set of some half-spaces,
 \beq\langle l_i, x\rangle < \lambda_i, ~i=1,...,d. \eeq
Here $l_i$ are $d-$vectors in $\mathbb {R}^n$ with all components in
$\mathbb {Z}$, which satisfy the Delzant condition ([Ab]). Without
loss of generality, we may assume that the original point $0$ lies
in $P$, so all $\lambda_i>0$.  Recall that a piecewise linear (PL)
function  $f$ on $P$ is of the form
$$u= \text{max}\{f^1,...,f^r\},$$
 where $f^\lambda=\sum a_i^\lambda x_i + c^\lambda, ~\lambda=1,...,r,$
for some vectors $(a^\lambda_i)\in \Bbb R^n$ and some numbers
$c^\lambda\in \Bbb R$. $f$ is called a rational PL-function if
components $a^\lambda_i$ and numbers $c^\lambda$ are all rational.
According to [D2],  a rational PL-function induces a test
configuration, called a toric degeneration.  Moreover, one can show
that a toric degeneration is always compatible to $\Bbb C^*-$actions
induced by  an extremal vector field on $M$  (cf. [ZZ1]).

Now we  can state our  first main theorem in this note.

\begin {theo} Let $M$ be a polarized toric manifold which admits an
extremal metric in this polarized K\"ahler class. Then $M$ is
$K$-stable relative to  $\Bbb C^*$-actions induced by an extremal
vector field on $M$ for any toric degeneration.
\end{theo}

 We will give a proof of Theorem 1.3 in the next section.
  The proof   can be also used to   discuss the properness of
the modified $K$-energy $\mu(\phi)$.  As same as the $K$-energy,
$\mu(\phi)$ is defined on a K\"ahler class whose critical point is
an extremal metric while the critical point of $K$-energy is a
K\"ahler metric with constant scalar curvature.

Let  $\omega_g$ be a K\"ahler form of K\"ahler metric $g$ on a
compact K\"ahler manifold $M$ and
$\omega_{\phi}=\omega_g+\sqrt{-1}\partial\overline\partial\phi$ be a
K\"ahler form associated to a potential function $\phi$ in K\"ahler
class $[\omega_g]$.  Let $\theta_X$ be a normalized potential
function associated to an extremal vector field $X$ and the metric
$\omega_g$ on $M$ [ZZ1].
 Then the $\mu(\phi)$ is given by
    $$\mu (\phi) = - \frac{1}{V}\int_{0}^{1}\int_{M}
\dot\phi_t[R(\omega_{\phi_t}) -\overline R-
\theta_X(\omega_{\phi_t})] \omega_{\phi_t}^n\wedge dt,$$
  where $\phi_t ~(0 \leq t \leq 1)$
is a family of potential functions connecting $0$ to $\phi$,
$R(\omega_{\phi_t})$ denote the scalar curvatures of
$\omega_{\phi_t}$,
 $\theta_X(\omega_{\phi_t})=\theta_X+X(\phi_t)$ are normalized potential
functions associated to  $X$ and  $\omega_{\phi_t}$, and $\overline
R$ is the average of scalar curvature of $\omega_g$. It can be
showed that the functional $\mu(\phi)$ is well-defined, i.e., it is
independent of the choice of path $\phi_t$ ([Gu]). Thus one sees
that  $\phi$ is a critical point of $\mu(\cdot)$ iff the
corresponding metric $\omega_\phi$ is extremal.

\begin {defi}   Let
$$I(\phi)=\frac{1}{V}\int_M \phi(\omega_g^n-\omega_{\phi}^n),$$
where $V=\int_M\omega_g^n$.   $\mu(\phi)$ is called proper
associated to a subgroup $G$ of the automorphisms group
$\text{Aut}(M)$ in K\"ahler class $[\omega_g]$  if there is a
continuous function $p(t)$ in $\Bbb R$ with the property
 $$ \lim_{t\to+\infty} p(t)=+\infty,$$
  such that
  $$\mu(\phi)\ge \inf_{\sigma\in G} p(I(\phi_{\sigma})),$$
   where
 $\phi_{\sigma}$ is defined by
$$\omega_g+\sqrt{-1}\partial\bar{\partial}\phi_{\sigma}
=\sigma^*(\omega_g+\sqrt{-1}\partial\bar{\partial}\phi).$$
\end{defi}

The above definition was first introduced by Tian for the
$K$-energy ([Ti]]. He proved that the properness of $K$-energy is
a sufficient and necessary condition for the existence of
K\"ahler-Einstein metrics. In a very recent paper [ZZ2], the
authors proved the existence of minimizing weak solution of
extremal metrics on toric manifolds under the assumption of
properness of $\mu(\phi)$. This weak solution will be an extremal
metric if one can further prove some regularities of the solution.

The following is our second main theorem in this note.

\begin {theo} Let $M$ be a toric surface which admits an
extremal metric $\omega_E$. Suppose that
 \beqn\bar{R}+\theta_X(\omega_E)>0.\eeqn
  Then $\mu(\phi)$ is proper
associated to  $T$ in the space of $G_0$-invariant K\"ahler metrics.
Here $T$ is a torus actions group of $M$ and $G_0$ is a maximal
compact subgroup of $T$.
\end{theo}

We note that in Theorem 1.5 we need not to assume that the surface
is polarized.

\section { Proof of theorem 1.3}

Let $P$ be an integral   polytope  in $\mathbb {R}^n$ associated to
a polarized toric manifold $M$  as in Section 1. We choose a
$G_0$-invariant K\"ahler metric $\omega_g$ on $M$. Then there exists
a convex function $\psi_0$ in $\Bbb R^n$ such that
$$\omega_g=\sqrt{-1}\partial\overline\partial \psi_0.$$
 Let  $u_0$ be a  Legendre function  of $\phi_0$ which is a convex function on
 $P$.
 Set
 $$\mathcal C= \{u |~u~\text{is a convex function on }~ P~\text{with}~
u-u_0\in C^\infty(\overline P)\}.$$
 Then one can show that functions in $\mathcal {C}$ are
 corresponding to $G_0$-invariant K\"ahler potential functions on $M$ by one-to-one (cf. [Ab],[D2]).

Let $d\sigma_0$ be the Lebesgue measure on the boundary $\partial
P$ and $\nu$ be the outer normal vector field on $\partial P$. Let
$d\sigma$ be an induced measure on $\partial P$ such that
$d\sigma=\lambda_i^{-1}(\nu, x)d\sigma_0$ on each face $\langle
l_i, x\rangle= \lambda_i$ of $P$.  We define a linear functional
on the space of continuous function on $\overline P$ by
$$L(u)=\int_{\partial P}u d\sigma-\int_P (\bar{R}+\theta_X)udx,$$
 where
 $\theta_X$ is the normalized potential function associated to the
extremal vector field $X$ as in Section 1, which is equal to
$<A,x>+a$ for some vector $A\in \Bbb R^n$ and constant $a$ in the
coordinates $x$ in $P$ ([ZZ1]).   The following result was proved in
[ZZ1].

 \begin{lem}[{[ZZ1]}]  Let $\tilde\alpha$ and $\tilde\beta$ are two $C^*$-actions induced by a rational
PL-function $f$ and the extremal vector field $X$.  Then
 \beqn
F_{\tilde{\beta}}(\tilde{\alpha})=-\frac{1}{2Vol(P)}L(f).\eeqn
\end{lem}

 By the above lemma, to prove   Theorem 1.3  we suffice to  prove

\begin {prop} Let $M$ be a  toric manifold which admits an
extremal metric. Then for any PL-function $f$ on $P$, we have
  \beqn
L(f)\geq 0.\eeqn
 Moreover the equality  holds if and only if $f$ is  an affine linear function on $P$.
\end{prop}

\begin {proof} By definition, we may assume that
$$f=\max\{f^1,..., f^r\},$$
and each  $f^\alpha=\sum c_i^\alpha x_i+c^\alpha$ is an affine
linear function on $P$. Then   $P$ can be divided into $m(\geq r)$
small polytopes $P^1,... , P^m$ such that for each $P^\lambda$
there exists a $f^{\alpha_\lambda}$ such that
$f=f^{\alpha_\lambda}$ on $P^\lambda$. By the assumption of the
existence of extremal metric, we see that there exists  a
$u\in\mathcal C$ which satisfies the Abreu's equation ([Ab]),
 \beq -\sum_{i,j}u_{ij}^{ij}=\bar{R}+\theta_X,\eeq
where $(u^{ij})=(u_{ij})^{-1}$ and $u_{kl}^{ij}=\frac{\partial^2
u^{ij}}{\partial x^k\partial x^l}$. Thus
 \beq L(f)=\int_{\partial P}f
d\sigma+\int_P \sum_{i,j}u^{ij}_{ij}fdx.\eeq

For any small $\delta> 0$,  we let $P_{\delta}$ be the interior
polygon with faces parallel to those of $P$ separated by distance
$\delta$. Integrating by parts, one sees
 \beqn &&\int_{P_{\delta}}
u^{ij}_{ij}f dx\notag\\
&& =\sum_\lambda\int_{P^\lambda\bigcap P_{\delta}} \sum_{i,j}u^{ij}_{ij}f^{\alpha_\lambda} dx\notag\\
&& =\sum_\lambda\int_{\partial(P^\lambda\bigcap P_{\delta})}
\sum_{i,j}u^{ij}_if^{\alpha_\lambda} n^{\alpha_\lambda}_jd\sigma_0
-\sum_\lambda\int_{P^{\lambda}\bigcap P_{\delta}} \sum_{i,j}u^{ij}_{i}c^{\alpha_\lambda}_j dx\notag\\
&&=\sum_\lambda\int_{\partial(P^{\lambda}\bigcap P_{\delta})}
\sum_{i,j}u^{ij}_if^{\alpha_\lambda} n^{\alpha_\lambda}_jd\sigma_0
-\sum_\lambda\int_{\partial(P^{\lambda}\bigcap P_{\delta})}
\sum_{i,j}u^{ij}c^{\alpha_\lambda}_j n^{\alpha_\lambda}_i d\sigma_0\notag\\
&&=\sum_\lambda\int_{\partial P_{\delta}}
\sum_{i,j}u^{ij}_if^{\alpha_\lambda}
n^{\alpha_\lambda}_jd\sigma_0-
\sum_\lambda\int_{\partial P_{\delta}}\sum_{i,j}u^{ij}c^{\alpha_\lambda}_j n^{\alpha_\lambda}_i d\sigma_0\notag\\
&&-\sum_{\lambda<\mu}\int_{\partial P^{\lambda}\bigcap
\partial P^{\mu}\bigcap P_\delta} \sum_{i,j}(u^{ij}c^{\alpha_\lambda}_j n^{\alpha_\lambda}_i +u^{ij}c^{\alpha_\mu}_j
n^{\alpha_\mu}_i)d\sigma_0,\eeqn
 where
$(n^{\alpha_\lambda}_1,...,n^{\alpha_\lambda}_n)$ is the unit
outer normal vector on $\partial P^{\lambda}$.  By [D2]  we know
that the first term goes to $-\int_{\partial P}f d\sigma$ and the
second term goes to $0$ at the last equality as $\delta$ goes to
zero. On the other hand, if $\partial P^{\lambda}\bigcap
\partial P^{\mu}$ is a common $(n-1)$-dimensional face of $P^\lambda$ and
$P^\mu$, then
$(c^{\alpha_\lambda}_1-c^{\alpha_\mu}_1,...,c^{\alpha_\lambda}_n-c^{\alpha_\mu}_n)$
is a nonzero vector, and
$$n^{\alpha_\lambda}_i=-\frac{1}{\sqrt{\sum_l(c^{\alpha_\lambda}_l-c^{\alpha_\mu}_l)^2}} (c^{\alpha_\lambda}_i-c^{\alpha_\mu}_i),
n^{\alpha_\mu}_i=-\frac{1}{\sqrt{\sum_l(c^{\alpha_\mu}_l-c^{\alpha_\lambda}_l)^2}}(c^{\alpha_\mu}_i-c^{\alpha_\lambda}_i).$$
Substituting  them into the third term  at the last equality in
(2.5) and letting $\delta$ go to $0$, we derive
  \beqn &&\int_P \sum_{i,j}u^{ij}_{ij}f
dx\notag\\
 &&=-\int_{\partial P}f
d\sigma+\sum_{\lambda<\mu}\int_{\partial P^\lambda\bigcap
\partial P^\mu}
\frac{\sum_{i,j}u^{ij}(c^{\alpha_\lambda}_i-c^{\alpha_\mu}_i)(c^{\alpha_\lambda}_j-c^{\alpha_\mu}_j)}{\sqrt{\sum_i(c^{\alpha_\lambda}_i
-c^{\alpha_\mu}_i)^2}}d\sigma_0.\eeqn
Hence, combining (2.4) and (2.6),
$$L(f)=\sum_{\lambda<\mu}\int_{\partial
P^\lambda\bigcap\partial P^\mu}
\frac{\sum_{i,j}u^{ij}(c^{\alpha_\lambda}_i-c^{\alpha_\mu}_i)(c^{\alpha_\lambda}_j-c^{\alpha_\mu}_j)}
{\sqrt{\sum_l(c^{\alpha_\lambda}_l-c^{\alpha_\mu}_l)^2}}d\sigma_0\geq
0.$$
  Note that the equality holds if only if there is no common
$(n-1)$-dimensional face for any $P^\lambda$ and $P^\mu$, which
implies that $f$ is just an  affine linear function.
\end {proof}

\section { Proof of Theorem 1.5}

To prove Theorem 1.5, we need to use a Donaldson's version of the
modified $K$-energy ([D2], [ZZ1]).

\begin {lem} Let $\omega_{\phi}$ be a $G_0$-invariant K\"ahler
metric on $M$ and $u$ be a Legendre function of $\phi$ in $\mathcal
{C}$. Then
$$\mu(\phi)=\frac{2^nn!(2\pi)^n}{Vol(M)} \mathcal {F}(u),$$
 where
\beq \mathcal {F}(u)=-\int_P \log(\det(D^2u))dx +L(u)\eeq
\end {lem}

 Let  $p\in P$. We set
 $$ \mathcal C_\infty =\{u\in C^0(\overline P)\cup C^\infty (P)|~u~ \text{is  a convex function on}
 ~P~  \text{with}~\inf_{P}u=u(p)=0 \}.$$
 The following proposition was proved in [ZZ1].

\begin {prop} Suppose that  there exists a $ \lambda> 0$ such that
\beqn L(u)  \geq \lambda \int_{\partial P}u d\sigma, \eeqn
  for any  $ u\in\mathcal C_\infty$.  Then there exist two uniform
  constants $\delta, C>0$ such that for any $G_0$-invariant K\"ahler
  metric $\omega_\phi$ it holds
  $$\mu(\phi)\ge  \delta \inf_{\sigma\in T} I(\phi_{\sigma})-C.$$
    In particular, $\mu(\phi)$ is proper associated to subgroup $T$ in the space of $G_0$-invariant
K\"ahler metrics.
\end{prop}

\begin{proof}[Proof of Theorem 1.5] According to Proposition 3.2, we suffice to verify  the  condition
(3.2). We use an argument by the contradiction. First note that by
the convexity of functional $\mathcal F(u)$ on $\mathcal C$, one
sees that  $\mathcal F(u)$ is bounded from below  by the existence
of extremal metrics.  Then   one can show (cf. [D2], [ZZ1]),
  \beqn L(u)  \geq 0, \forall~u\in\mathcal C_\infty.\eeqn
  Furthermore one concludes  that
 \beqn L(u)  \geq 0, ~\forall~ u\in\mathcal C_1,\eeqn
  where $\mathcal C_1$ is a set of positive convex functions $u$ on
  $P\cup\partial P$ such that
  $$\int_{\partial P} ud\sigma_0 dx<\infty.$$
On the other hand,  suppose that (3.2) is not true, then by (3.4) it
is easy to see that there exists a $u_\infty\in \mathcal C_1$, which
is not affine linear,  such that
 \beqn L(u_\infty)  = 0.\eeqn.

Recall that a simple PL-function is a form of
$$u=\max \{0,\sum a_i x_i +c\}$$
for some vector  $(a_i) \in \Bbb R^n$ and number $c \in \Bbb R$.
 We call the hyperplane $\sum a_i x_i +c=0$  a crease of $u$.
Then applying  Proposition 5.3.1 in [D2], we see that under the
relations (3.4), (3.5) and  the assumption (1.3) in Theorem 1.5 for
the case of toric surfaces there exists a simple PL-function $v_0$
with crease intersecting the interior of $P$  such that $L(v_0)= 0$.
But the late is contradict to Proposition 2.3. Thus Proposition 3.2
is true and so is  theorem 1.5.

\end{proof}


\begin{thebibliography}{Gua}

\bibitem[Ab]{Ab}
 Abreu, M.,  K\"ahler geometry of
toric varieties and extremal metrics,   Inter. J. Math. Vol 9
(1998), 641-651.


\bibitem [D1] {Do} Donaldson, S.K., Scalar curvature and projective
embeddings, I,  J. Diff. Geom., Vol 59 (2001),  479-522.

\bibitem [D2] {Do} Donaldson, S.K., Scalar curvature and
stability of toric varieties,  J. Diff. Geom., Vol 62 (2002),
289-349.


\bibitem[Gu]{ Gu}  Guan, D., On modified Mabuchi functional and Mabuchi
moduli space of K\"ahler metrics on toric bundle,  Math. Res.
Lett., Vol 6 (1999), 547-555.



\bibitem  [M1]{M1} Mabuchi, T., An energy-theoretic approach to the Hitchin-Kobayashi
correspondence for manifolds, I,  Invent. Math., Vol 159 (2005),
225-243.


\bibitem [M2] { M2} Mabuchi, T.,  An energy-theoretic approach to the Hitchin-Kobayashi
correspondence for manifolds, II ,  preprint, 2004.



\bibitem[Sz]{ Sz} Sz\'{e}kelyhidi, G., Extremal
metrics and K-stability,  Bull. London Math. Soc., V ol 39(2007),
76-84.


\bibitem[Ti] {Ti} Tian, G., K\"ahler-Einstein metrics with
positive scalar curvature,  Invent. Math., V ol 130 (1997), 1-39.

\bibitem[Ya]{ Ya}Yau, S.T.,  Open problem in geometry,  Proc. of Symp. in Pure
Math., Vol 54 (1993),  1-28.

\bibitem[ZZ1]{ZZ1}
Zhou, B. and Zhu, X.H., Relative K--stability and modified K-energy
on toric manifolds, preprint, 2006.

\bibitem[ZZ2]{ZZ2} Zhou, B. and Zhu, X.H., Minimizing  weak solutions for calabi's extremal metrics on  toric
manifolds, preprint, 2006.


\end{thebibliography}
\end{document}